\documentclass[12pt]{article}

\usepackage{amsmath,amssymb,amsthm,amscd,dsfont}
\usepackage{amsfonts,latexsym,rawfonts}
\usepackage{lscape}
\usepackage{amscd, float,times,rotating}
\usepackage{pb-diagram}
\usepackage{hyperref}
\numberwithin{equation}{section}

\newtheorem{theorem}{Theorem}[section]
\newtheorem{lemma}[theorem]{Lemma}

\newtheorem{corollary}[theorem]{Corollary}
\theoremstyle{definition}

\newtheorem{example}[theorem]{Example}
\theoremstyle{remark}
\newtheorem{remark}[theorem]{Remark}

\newcommand{\R}{\mathbb{R}}
\newcommand{\Sph}{\mathbb{S}}
\renewcommand{\Cap}{\operatorname{Cap}}
\newcommand{\Sec}{\operatorname{Sec}}

\newcommand{\tr}{\operatorname{tr}}
\newcommand{\Id}{\operatorname{Id}}

\title{Curvature-Based Bounds for $p$-Capacity and Capacity--Area Inequalities}
\author{Xiaoshang Jin\\[4pt]
\small School of Mathematics and Statistics, Huazhong University of Science and Technology,\\
\small Wuhan 430074, P.R. China\\
\small \texttt{jinxs@hust.edu.cn}}
\date{}

\begin{document}

\maketitle

\begin{abstract}
We establish curvature-based bounds for the $p$-capacity of convex domains and related capacity--area inequalities. Let $K$ denote the Gauss--Kronecker curvature of $\partial\Omega$. In Euclidean space, if
$
0<K_-\leq K\leq K_+,
$
then, for every $1<p\leq \frac{n+1}{2}$,
$$
\operatorname{Cap}_p\left(B(K_+^{\frac{1}{1-n}})\right)
\leq
\operatorname{Cap}_p(\Omega)
\leq
\operatorname{Cap}_p\left(B(K_-^{\frac{1}{1-n}})\right).
$$
The lower comparison remains valid for the full range $1<p<n$. Moreover, the exponent range in the upper comparison is optimal in general: when $n\geq 4$ and $\frac{n+1}{2}<p<n$, the upper comparison may fail. Thus $p=\frac{n+1}{2}$ is the critical exponent for the curvature-radius upper comparison.

Under the pinching condition
$
\frac{K_-}{K_+}=\rho\in (0, 1],
$
we further obtain an explicit scale-invariant capacity--area lower bound for every $1<p<n$. In the special case $n=3$ and $p=2$,
$$
\frac{\operatorname{Cap}_2(\Omega)}
{\sqrt{|\partial\Omega|}}
\geq
\frac{4\sqrt{\pi}\,\rho^{1/4}}
{1+\sqrt{\rho}},
$$
which implies the P\'olya--Szeg\"o inequality whenever
$$
\rho\geq \rho_*:=
\left(
\frac{\pi^2-4-\pi\sqrt{\pi^2-8}}{4}
\right)^2
\simeq 0.15484.
$$

On a Cartan--Hadamard manifold with
$\operatorname{Sec}\leq -k^2<0$, we obtain a $p$-capacity lower bound retaining all principal curvatures of the boundary, together with corresponding Gauss--Kronecker curvature estimates and a rigidity characterization of the equality case by a warped-product exterior geometry.
\end{abstract}

\medskip
\noindent\textbf{2020 Mathematics Subject Classification.}
Primary 31B15, 35J92; Secondary 53C21, 53C23.

\smallskip
\noindent\textbf{Keywords.}
$p$-capacity, exterior $p$-harmonic potential, Gauss--Kronecker curvature,
Cartan--Hadamard manifold, P\'olya--Szeg\"{o} conjecture.

\section{Introduction}

Let $(M^n,g)$ be a complete and non-compact Riemannian manifold and let $E\subset M$ be compact. For $p>1$, the $p$-capacity of $E$ is defined by
\begin{equation}\label{eq:intro-capacity}
 \Cap_p(E)
 :=\inf\left\{\int_M |\nabla\varphi|^p\,dV_g:
 \varphi\in C_c^\infty(M),\ \varphi\ge1\ \text{on }E\right\}.
\end{equation}
When $E=\overline\Omega$ is the closure of a smooth bounded domain and the infimum is attained, the minimizer is the exterior $p$-capacitary potential. It solves
\begin{equation}\label{eq:intro-potential}
 \begin{cases}
 \operatorname{div}_g\bigl(|\nabla u|^{p-2}\nabla u\bigr)=0
   & \text{in }M\setminus\overline\Omega,\\
 u=1 & \text{on }\partial\Omega,\\
 u\to0 & \text{at infinity},
 \end{cases}
\end{equation}
and its Dirichlet energy equals $\Cap_p(\Omega)$.

The concept of capacity originates in electrostatics. For general $p$, it is a basic quantity in nonlinear potential theory and geometric analysis. In Euclidean space, relations between capacity, volume, surface area, and curvature go back to the classical work of P\'olya and Szeg\"{o} and have been developed in many directions; see, for example, \cite{PolyaSzego,Mazya,Xiao,FogagnoloMazzieriPinamonti,ColesantiNystromSalaniXiaoYangZhang}. On complete Riemannian manifolds, capacity is closely related to isoperimetric inequalities, parabolicity, trace inequalities, and the geometry at infinity; see \cite{Grigoryan,FogagnoloMazzieri,HurtadoPalmerRitore}.

The Euclidean part of the paper is organized around a two-sided comparison with model balls determined only by the extrema of the Gauss--Kronecker curvature. The lower comparison comes from a one-dimensional normal-ray estimate that retains the boundary curvature pointwise. The upper comparison is more subtle. Its key input is a mixed curvature--field monotonicity obtained by combining a direct exterior consequence of the height-concavity theorem of Ma and Zhang \cite{MaZhang} with the first variation of capacitary level-set area and the conserved $p$-harmonic flux. At the critical exponent $p=\frac{n+1}{2}$ this monotonicity becomes purely geometric; using the corresponding critical capacitary potential as a test function then propagates the upper comparison to every $1<p\le\frac{n+1}{2}$. Thus the mixed curvature--field estimate is used below as a technical engine, while the main Euclidean conclusion is stated directly in terms of $K_-$ and $K_+$.

This curvature-radius comparison is structurally different from the surface-area estimates in \cite{Xiao}: it contains no boundary gradient and no area term, and in general it does not arise from a set-theoretic inclusion between $\Omega$ and the two model balls. The normal-ray argument also gives a Gauss--Kronecker-curvature integral lower bound, from which the explicit pinching theorem below follows. No optimality is claimed for the pinching constant when $0<\rho<1$.

Throughout the Euclidean part, $K=\kappa_1\cdots\kappa_{n-1}$ denotes the Gauss--Kronecker curvature of $\partial\Omega$ with respect to the outward unit normal, where $\kappa_1,\ldots,\kappa_{n-1}$ are the principal curvatures. We write
\[
 K_-=\min_{\partial\Omega}K,
 \qquad
 K_+=\max_{\partial\Omega}K.
\]
For a Euclidean ball $B(R)\subset\mathbb R^n$ and $1<p<n$,
\begin{equation}\label{eq:intro-ball-capacity}
 \Cap_p(B(R))
 =|\mathbb S^{n-1}|
 \left(\frac{n-p}{p-1}\right)^{p-1}R^{n-p}.
\end{equation}

We first state the Gauss-curvature comparison.

\begin{theorem}\label{thm:intro-gauss-comparison}
Let $n\ge3$, let $\Omega\subset\mathbb R^n$ be a smooth bounded convex domain, and assume that
\[
 0<K_-\le K\le K_+
 \qquad\text{on }\partial\Omega.
\]
Then for every $1<p\le\frac{n+1}{2}$,
\begin{equation}\label{eq:intro-gauss-comparison}
 \Cap_p\!\left(B(K_+^{\frac{1}{1-n}})\right)
 \le \Cap_p(\Omega)
 \le \Cap_p\!\left(B(K_-^{\frac{1}{1-n}})\right).
\end{equation}
The left-hand inequality remains valid for $1<p<n$. Equality in the left-hand inequality holds if and only if $\Omega$ is a Euclidean ball, and for every $1<p\le\frac{n+1}{2}$ the same is true for the right-hand inequality.
\end{theorem}

The inequalities in \eqref{eq:intro-gauss-comparison} do not in general follow from set-theoretic inclusions between $\Omega$ and the two model balls. Moreover, the restriction $p\le\frac{n+1}{2}$ in the upper comparison is not merely an artifact of the proof: when $n\ge4$ and $\frac{n+1}{2}<p<n$, the right-hand inequality in \eqref{eq:intro-gauss-comparison} need not hold. Thus the critical exponent $p=\frac{n+1}{2}$ is optimal in general for this curvature-radius upper comparison. Both the failure of ball inclusions and the supercritical obstruction are exhibited by the rotational ellipsoids in Example~\ref{ex:ellipsoid-obstructions}.

In dimension three, the classical P\'olya--Szeg\"{o} conjecture asserts that every convex body satisfies, in the present unnormalized convention,
\[
 \frac{\Cap_2(\Omega)^2}{|\partial\Omega|}\ge\frac{32}{\pi}.
\]
The sharp constant is associated with the degenerate disk model. For earlier progress on the conjecture, including shape-optimization results, a local version of the conjectured inequality, and related $p$-capacity--surface-area estimates, see \cite{CrastaFragalaGazzola,FragalaGazzolaPierre,Xiao}. The next result is a general $p$-capacity pinching estimate. It is scale invariant for every $1<p<n$, and the P\'olya--Szeg\"{o} conjecture appears as a special case when $n=3$ and $p=2$.

\begin{theorem}\label{thm:intro-gauss-pinching}
Let $n\ge3$, let $1<p<n$, and let $\Omega\subset\mathbb R^n$ be a smooth bounded convex domain. Assume that its Gauss--Kronecker curvature satisfies
\[
 0<K_-\le K\le K_+,
 \qquad
 \frac{K_-}{K_+}=\rho\in(0,1].
\]
Then
\begin{equation}\label{eq:intro-gauss-pinching}
 \frac{\Cap_p(\Omega)^{n-1}}{|\partial\Omega|^{n-p}}
 \ge\Phi_{p,n}(\rho)
\end{equation}
where
\[
\Phi_{p,n}(\rho)=\begin{cases}
|\mathbb S^{n-1}|^{p-1}\bigl(\frac{n-p}{p-1}\bigr)^{(n-1)(p-1)},
& \rho=1;\\[1ex]
\displaystyle
\begin{aligned}
&|\mathbb S^{n-1}|^{p-1}
 \frac{(n-1)^{n-1}(n-p)^{n(p-2)+1}}
 {(p-1)^{n(p-1)}}\\[-0.2ex]
&\quad\times
 \frac{\rho^{\frac{(p-1)(n-p)}{n-1}}
 \left(1-\rho^{\frac{p-1}{n-1}}\right)^{p-1}}
 {(1-\rho)^{n-1}}
 \left(1-\rho^{\frac{n-p}{n-1}}\right)^{n-p},
\end{aligned}
& 0<\rho<1.
\end{cases}
\]

In particular, when $n=3$ and $p=2$,
\begin{equation}\label{eq:intro-newtonian-pinching}
 \frac{\Cap_2(\Omega)}{\sqrt{|\partial\Omega|}}
 \ge
 \frac{4\sqrt\pi\,\rho^{1/4}}{1+\sqrt\rho}.
\end{equation}
Hence, if
\begin{equation}\label{eq:intro-ps-improved-threshold}
 \rho\ge\rho_*:=
 \left(\frac{\pi^2-4-\pi\sqrt{\pi^2-8}}{4}\right)^2
 \simeq 0.15484,
\end{equation}
then $\Omega$ satisfies the P\'olya--Szeg\"{o} inequality
\[
 \frac{\Cap_2(\Omega)^2}{|\partial\Omega|}\ge\frac{32}{\pi}.
\]
\end{theorem}

Theorem~\ref{thm:intro-gauss-comparison} is proved by combining two results established later: the mixed curvature--field monotonicity in Section~\ref{sec:euclidean} and the normal-ray curvature lower bound in Section~\ref{sec:euclidean-ray}. Theorem~\ref{thm:intro-gauss-pinching} uses the latter curvature integral together with the concavity of $x^{(p-1)/(n-1)}$ and the Gauss-map area formula. At the critical exponent $p=\frac{n+1}{2}$, taking the $(n-1)$-st root of the pinching estimate yields the factor $2\rho^{1/4}/(1+\sqrt\rho)$.

In asymptotically flat three-manifolds with nonnegative scalar curvature, Bray and Miao obtained a capacity upper bound in terms of boundary area and the Willmore functional \cite{BrayMiao}. Recent related developments in non-Euclidean capacity include conformal and flow methods in hyperbolic space \cite{LiLiXiong,JinXiao,JinWanXiao}, capacity--volume estimates on rotationally symmetric manifolds \cite{JinXiaoEssential}, and comparison estimates under lower Ricci-curvature bounds \cite{JinLuYin}.

We next state the Cartan--Hadamard theorem. For the classical case $p=2$, capacity comparison inequalities on Cartan--Hadamard manifolds under lower principal-curvature bounds were obtained by Hurtado, Palmer and Ritor\'e \cite{HurtadoPalmerRitore}. Our estimate below treats every $p>1$ and retains the full initial shape operator. Under the curvature assumption $\Sec_N\le-k^2<0$, the ambient manifold is $p$-hyperbolic for every $p>1$; see \cite[Section~6.2]{DAmbrosioDipierro}. Related asymptotic Dirichlet theory for the $p$-Laplacian on negatively pinched Cartan--Hadamard manifolds is developed by Holopainen \cite{Holopainen}. Standard $p$-potential theory then provides the exterior $p$-capacitary potential for smooth relatively compact domains, for instance by exhaustion and comparison with a positive $p$-Green function; see also \cite{FogagnoloMazzieri}.

For $k>0$ and nonnegative numbers $\kappa_1,\ldots,\kappa_{n-1}$, define
\begin{equation}\label{eq:intro-ch-full-constant}
 \mathcal C_{n,p}(k;\kappa_1,\ldots,\kappa_{n-1})
 :=\left[
 \int_0^\infty
 \prod_{i=1}^{n-1}
 \left(\cosh(ks)+\frac{\kappa_i}{k}\sinh(ks)\right)^{-\frac1{p-1}}
 \,ds
 \right]^{1-p}.
\end{equation}
For $\tau\ge0$, we also keep the scalar model notation
\begin{equation}\label{eq:intro-ch-constant}
 C_{n,p}(k,\tau)
 :=\left[
 \int_0^\infty
 \left(\cosh(ks)+\frac{\tau}{k}\sinh(ks)\right)^{-\frac{n-1}{p-1}}\,ds
 \right]^{1-p}.
\end{equation}

\begin{theorem}\label{thm:intro-ch}
Let $(N^n,g)$ be a Cartan--Hadamard manifold satisfying $\Sec_N\le-k^2<0$, and let $\Omega\Subset N$ be smooth and geodesically convex. If $\kappa_1(y),\ldots,\kappa_{n-1}(y)$ are the principal curvatures of $\partial\Omega$ with respect to the outward normal, then for every $p>1$,
\begin{align}
 \Cap_p(\Omega)
 &\ge
 \int_{\partial\Omega}
 \mathcal C_{n,p}\bigl(k;\kappa_1(y),\ldots,\kappa_{n-1}(y)\bigr)\,dA_g(y)
 \label{eq:intro-ch-full}\\
 &\ge
 \int_{\partial\Omega}
 C_{n,p}\bigl(k,K(y)^{1/(n-1)}\bigr)\,dA_g(y)
 \label{eq:intro-ch-gauss}\\
 &\ge
 \int_{\partial\Omega}
 C_{n,p}\left(k,\min_i\kappa_i(y)\right)\,dA_g(y).
 \label{eq:intro-ch}
\end{align}
Here $K=\prod_{i=1}^{n-1}\kappa_i$ is the Gauss--Kronecker curvature of $\partial\Omega$.
\end{theorem}

\begin{corollary}\label{cor:intro-ch-uniform}
Under the assumptions of Theorem~\ref{thm:intro-ch}, if
\[
 \kappa_i\ge\lambda_0\ge0,
 \qquad i=1,\ldots,n-1,
\]
then
\begin{equation}\label{eq:intro-ch-uniform}
 \Cap_p(\Omega)
 \ge C_{n,p}(k,\lambda_0)|\partial\Omega|_g.
\end{equation}
Equality holds if and only if the exterior $N\setminus\Omega$ is isometric to
\begin{equation}\label{eq:intro-ch-rigidity-metric}
 \left([0,\infty)\times\partial\Omega,
 ds^2+
 \left(\cosh(ks)+\frac{\lambda_0}{k}\sinh(ks)\right)^2
 g_{\partial\Omega}\right),
\end{equation}
with $\partial\Omega$ totally umbilical of principal curvature $\lambda_0$ and all radial sectional curvatures equal to $-k^2$. In this case the exterior potential is
\begin{equation}\label{eq:intro-ch-rigidity-potential}
 u(s)=
 \frac{\displaystyle\int_s^\infty
 \left(\cosh(k\tau)+\frac{\lambda_0}{k}\sinh(k\tau)\right)^{-\frac{n-1}{p-1}}d\tau}
 {\displaystyle\int_0^\infty
 \left(\cosh(k\tau)+\frac{\lambda_0}{k}\sinh(k\tau)\right)^{-\frac{n-1}{p-1}}d\tau}.
\end{equation}
In the constant-curvature space form, equality therefore forces $\Omega$ to be a geodesic ball.
\end{corollary}

The paper is organized as follows. After the preliminaries, Section~\ref{sec:euclidean} develops the Euclidean capacitary level-set argument and proves the mixed curvature--field monotonicity that drives the upper comparison. Section~\ref{sec:euclidean-ray} proves the normal-ray curvature lower bound, then combines the two Euclidean mechanisms to prove Theorem~\ref{thm:intro-gauss-comparison} and establishes the Gauss-curvature pinching estimate in Theorem~\ref{thm:intro-gauss-pinching}. Section~\ref{sec:ch} establishes the full-shape-operator Cartan--Hadamard estimate and the associated rigidity statement.
\section{Preliminaries}\label{sec:preliminaries}

\subsection{Capacity and exterior potentials}

We use the unnormalized convention \eqref{eq:intro-capacity}. If the exterior potential $u$ in \eqref{eq:intro-potential} exists, then
\begin{equation}\label{eq:energy-capacity}
 \Cap_p(\Omega)
 =\int_{M\setminus\Omega}|\nabla u|^p\,dV_g.
\end{equation}
If $t\in(0,1)$ is a regular value and $\Sigma_t=\{u=t\}$, the $p$-harmonic flux is constant:
\begin{equation}\label{eq:general-flux}
 \Cap_p(\Omega)
 =\int_{\Sigma_t}|\nabla u|^{p-1}\,dA_g.
\end{equation}
Indeed, the divergence theorem applied to the region between two regular level hypersurfaces shows that the right-hand side is independent of $t$, and the equality with the energy follows by testing the equation with $u$ and passing to infinity.

For the Euclidean part, the restriction $1<p<n$ guarantees the familiar polynomially decaying exterior potential. In the negatively curved part, the arguments below are variational and apply for every $p>1$ because the relevant one-dimensional model weights have exponential growth. In particular, a Cartan--Hadamard manifold with $\Sec\le-k^2<0$ is $p$-hyperbolic for every $p>1$; see \cite[Section~6.2]{DAmbrosioDipierro}. The standard exhaustion construction, together with comparison by a positive $p$-Green function, then yields the exterior $p$-capacitary potential of any smooth relatively compact domain; see \cite{FogagnoloMazzieri}.

Throughout the paper, principal curvatures are computed with respect to the outward normal and are positive on Euclidean spheres and on the boundaries of geodesic balls. If $A$ is the corresponding shape operator, then the mean curvature is $H=\tr A$ and the Gauss--Kronecker curvature is $K=\det A$.

\subsection{Weighted ray inequality}

The following elementary lemma is the analytic core of the Cartan--Hadamard normal-ray argument.

\begin{lemma}\label{lem:weighted-ray}
Let $p>1$ and let $J:(0,\infty)\to(0,\infty)$ be measurable. Assume that
\[
 \int_0^\infty J(s)^{-1/(p-1)}\,ds<\infty.
\]
If $f$ is absolutely continuous on compact intervals, $f(0)\ge1$, and $f(s)\to0$ as $s\to\infty$, then
\begin{equation}\label{eq:weighted-ray}
 \int_0^\infty |f'(s)|^pJ(s)\,ds
 \ge
 \left(\int_0^\infty J(s)^{-1/(p-1)}\,ds\right)^{1-p}.
\end{equation}
Equality is attained by
\begin{equation}\label{eq:weighted-extremal}
 f(s)
 =\frac{\displaystyle\int_s^\infty J(\tau)^{-1/(p-1)}\,d\tau}
 {\displaystyle\int_0^\infty J(\tau)^{-1/(p-1)}\,d\tau}.
\end{equation}
\end{lemma}

\begin{proof}
Since $f(0)\ge1$ and $f(\infty)=0$, the fundamental theorem of calculus gives
\[
 1\le\int_0^\infty |f'(s)|\,ds.
\]
Applying H\"older's inequality to $|f'|J^{1/p}$ and $J^{-1/p}$, we obtain
\[
 1\le
 \left(\int_0^\infty |f'|^pJ\,ds\right)^{1/p}
 \left(\int_0^\infty J^{-1/(p-1)}\,ds\right)^{(p-1)/p}.
\]
Raising both sides to the power $p$ proves \eqref{eq:weighted-ray}. Direct differentiation and substitution show that \eqref{eq:weighted-extremal} gives equality.
\end{proof}

\section{The Euclidean mixed-curvature estimate}\label{sec:euclidean}

Throughout this section, let $\Omega\subset\R^n$ be smooth, bounded, and strictly convex, let $1<p<n$, and let $u$ solve
\begin{equation}\label{eq:euclidean-potential}
 \begin{cases}
 \operatorname{div}\bigl(|\nabla u|^{p-2}\nabla u\bigr)=0
   & \text{in }\R^n\setminus\overline\Omega,\\
 u=1 & \text{on }\partial\Omega,\\
 u(x)\to0 & \text{as }|x|\to\infty.
 \end{cases}
\end{equation}
The convexity and regularity of the level hypersurfaces are classical consequences of Lewis' work \cite{Lewis}.

Define
\begin{equation}\label{eq:intro-lambda}
 \lambda
 :=\min_{\partial\Omega}
 \left(|\nabla u|^{\,n+1-2p}K\right)^{1/(n-1)}\ge0,
\end{equation}
where $K$ is the Gauss--Kronecker curvature of $\partial\Omega$ with respect to the outward normal. We allow $\lambda=0$; in that case the lower endpoint estimate below is trivial and the upper one is understood in the limiting extended sense $\lambda\downarrow0$.

\subsection{Level-set identities and propagation of the mixed curvature quantity}

For $0<t\le1$, set $\Omega_t=\{u\ge t\}$ and $\Sigma_t=\{u=t\}$. The outward unit normal to $\Omega_t$ is
\[
 \nu=-\frac{\nabla u}{|\nabla u|}.
\]
Let $H_t$ and $K_t$ denote respectively the mean curvature and the Gauss--Kronecker curvature of $\Sigma_t$ with respect to this normal.

\begin{lemma}\label{lem:level-identities}
For every regular value $t\in(0,1)$,
\begin{equation}\label{eq:level-identities}
 \Cap_p(\Omega)=\int_{\Sigma_t}|\nabla u|^{p-1}\,dA,
 \qquad
 -\frac{d}{dt}|\Sigma_t|
 =\int_{\Sigma_t}\frac{H_t}{|\nabla u|}\,dA.
\end{equation}
The second identity holds for almost every $t\in(0,1)$.
\end{lemma}

\begin{proof}
The first identity is the Euclidean case of \eqref{eq:general-flux}. For the second one, let $x(t)$ be a local parametrization of the level hypersurfaces. Differentiating $u(x(t))=t$ gives $\nabla u\cdot x'(t)=1$. Since $\nabla u=-|\nabla u|\nu$, the normal velocity with respect to $t$ is $-|\nabla u|^{-1}$. The first variation formula for area therefore gives the second identity in \eqref{eq:level-identities}.
\end{proof}

The exponent $n+1-2p$ in $|\nabla u|^{n+1-2p}K$ is not chosen ad hoc; it is precisely the exponent appearing in the height-concavity theorem of Ma and Zhang \cite{MaZhang}. The following propagation statement is a direct exterior consequence of their theorem; we include the short argument to record the normalization and the passage to infinity.

\begin{lemma}\label{lem:propagation}
Let $\lambda$ be defined by \eqref{eq:intro-lambda}. Then for every $0<t<1$,
\begin{equation}\label{eq:mixed-propagation}
 |\nabla u|^{\,n+1-2p}K_t\ge(\lambda t)^{n-1}
 \qquad\text{on }\Sigma_t.
\end{equation}
\end{lemma}

\begin{proof}
If $\lambda=0$, the assertion reduces to the nonnegativity of the Gauss--Kronecker curvature of the convex level hypersurfaces and is immediate. We may therefore assume $\lambda>0$. Fix $\varepsilon\in(0,t)$ and consider the convex ring bounded by $\partial\Omega$ and $\Sigma_\varepsilon$. Normalize the potential by
\[
 v_\varepsilon=\frac{u-\varepsilon}{1-\varepsilon}.
\]
Thus $v_\varepsilon=1$ on $\partial\Omega$ and $v_\varepsilon=0$ on $\Sigma_\varepsilon$. Ma and Zhang prove that
\[
 q_\varepsilon(a)
 :=\min_{\{v_\varepsilon=a\}}
 \left(|\nabla v_\varepsilon|^{\,n+1-2p}K\right)^{1/(n-1)}
\]
is concave for interior height values. By the boundary regularity of the capacitary potential and the continuity of the curvature quantity, $q_\varepsilon$ extends continuously to $[0,1]$, so the same concavity inequality holds up to the endpoints. The hypersurface $\Sigma_t$ is the level $v_\varepsilon=a_t$ with $a_t=(t-\varepsilon)/(1-\varepsilon)$. Concavity places the graph above the chord joining $(0,q_\varepsilon(0))$ and $(1,q_\varepsilon(1))$, hence
\[
 q_\varepsilon(a_t)
 \ge (1-a_t)q_\varepsilon(0)+a_tq_\varepsilon(1)
 \ge a_tq_\varepsilon(1),
\]
because $q_\varepsilon(0)\ge0$. Therefore, on $\Sigma_t$,
\[
 \left(|\nabla v_\varepsilon|^{\,n+1-2p}K_t\right)^{1/(n-1)}
 \ge
 \frac{t-\varepsilon}{1-\varepsilon}
 \min_{\partial\Omega}
 \left(|\nabla v_\varepsilon|^{\,n+1-2p}K\right)^{1/(n-1)}.
\]
Since $|\nabla v_\varepsilon|=|\nabla u|/(1-\varepsilon)$, the common scaling factor cancels. Letting $\varepsilon\downarrow0$ proves \eqref{eq:mixed-propagation}.
\end{proof}

\subsection{Mixed curvature--field monotonicity and the critical exponent}

The next theorem gives the mixed curvature--field monotonicity that drives the Euclidean upper comparison.

\begin{theorem}\label{thm:mixed-euclidean}
Let $\Omega\subset\mathbb R^n$ be smooth, bounded, and strictly convex, and let $1<p<n$. The function
\begin{equation}\label{eq:intro-monotone}
 t\longmapsto
 \left(\frac{\Cap_p(\Omega)}{|\Sigma_t|}\right)^{
 \frac{2(n-p)}{(n-1)(p-1)}}
 -\frac{n-p}{p-1}\lambda t^2
\end{equation}
is nondecreasing on $(0,1]$. Consequently,
\begin{equation}\label{eq:intro-euclidean-bounds}
 |\partial\Omega|
 \left(\frac{n-p}{p-1}\lambda\right)^{
 \frac{(n-1)(p-1)}{2(n-p)}}
 \le \Cap_p(\Omega)
 \le |\mathbb S^{n-1}|
 \left(\frac{n-p}{(p-1)\lambda}\right)^{\frac{n-1}{2}}.
\end{equation}
If $n\ge3$, equality in the lower estimate holds if and only if $\Omega$ is a Euclidean ball. Euclidean balls attain all the displayed constants.
\end{theorem}

\begin{proof}[Proof of Theorem~\ref{thm:mixed-euclidean}]
Taking the $(n-1)$-st root in Lemma~\ref{lem:propagation} and using the arithmetic--geometric mean inequality gives
\[
 H_t\ge(n-1)K_t^{1/(n-1)}
 \ge(n-1)\lambda t\,|\nabla u|^{(2p-n-1)/(n-1)}.
\]
Together with the first-variation identity in \eqref{eq:level-identities}, this yields the integral appearing below. The function
\[
 s\longmapsto s^{-\frac{2(n-p)}{(n-1)(p-1)}}
\]
is convex on $(0,\infty)$. Jensen's inequality, applied to $|\nabla u|^{p-1}$ on $\Sigma_t$, together with the flux identity, yields
\begin{align}
 \int_{\Sigma_t}|\nabla u|^{-2(n-p)/(n-1)}\,dA
 &=\int_{\Sigma_t}
 \bigl(|\nabla u|^{p-1}\bigr)^{-\frac{2(n-p)}{(n-1)(p-1)}}\,dA\notag\\
 &\ge |\Sigma_t|^{1+\frac{2(n-p)}{(n-1)(p-1)}}
 \left(\int_{\Sigma_t}|\nabla u|^{p-1}\,dA\right)^{-\frac{2(n-p)}{(n-1)(p-1)}}\notag\\
 &=|\Sigma_t|^{1+\frac{2(n-p)}{(n-1)(p-1)}}
 \Cap_p(\Omega)^{-\frac{2(n-p)}{(n-1)(p-1)}}.
 \label{eq:jensen-direct}
\end{align}
Combining this Jensen estimate with the preceding curvature bound and the first-variation identity gives
\begin{equation}\label{eq:area-riccati}
 -\frac{d}{dt}|\Sigma_t|
 \ge(n-1)\lambda t\,
 \Cap_p(\Omega)^{-\frac{2(n-p)}{(n-1)(p-1)}}
 |\Sigma_t|^{1+\frac{2(n-p)}{(n-1)(p-1)}}.
\end{equation}
Consequently,
\begin{align*}
 &\frac{d}{dt}\left[
 \left(\frac{\Cap_p(\Omega)}{|\Sigma_t|}\right)^{
 \frac{2(n-p)}{(n-1)(p-1)}}
 -\frac{n-p}{p-1}\lambda t^2
 \right]\\
 &\quad=-\frac{2(n-p)}{(n-1)(p-1)}
 \Cap_p(\Omega)^{\frac{2(n-p)}{(n-1)(p-1)}}
 |\Sigma_t|^{-1-\frac{2(n-p)}{(n-1)(p-1)}}
 \frac{d}{dt}|\Sigma_t|
 -\frac{2(n-p)}{p-1}\lambda t\\
 &\quad\ge0.
\end{align*}
This proves the monotonicity in \eqref{eq:intro-monotone}.

Since $|\Sigma_t|\to\infty$ as $t\downarrow0$, the function in \eqref{eq:intro-monotone} tends to zero as $t\downarrow0$. Evaluating its monotonicity at $t=1$ gives
\[
 \left(\frac{\Cap_p(\Omega)}{|\partial\Omega|}\right)^{
 \frac{2(n-p)}{(n-1)(p-1)}}
 \ge\frac{n-p}{p-1}\lambda.
\]
This is the lower bound in \eqref{eq:intro-euclidean-bounds}.

For the upper bound, we use the standard far-field asymptotics of the exterior $p$-capacitary potential and its gradient. In the present unnormalized convention, Lemma~2.15 of Colesanti--Nystr\"om--Salani--Xiao--Yang--Zhang \cite{ColesantiNystromSalaniXiaoYangZhang} (see also \cite[Proposition~2.4]{FogagnoloMazzieriPinamonti}) gives
\begin{equation}\label{eq:far-field}
\begin{aligned}
 u(x)
 &=A|x|^{-\frac{n-p}{p-1}}+o\left(|x|^{-\frac{n-p}{p-1}}\right),\\
 |\nabla u(x)|
 &=\frac{n-p}{p-1}A|x|^{-\frac{n-1}{p-1}}
   +o\left(|x|^{-\frac{n-1}{p-1}}\right).
\end{aligned}
\end{equation}
where
\[
 A=\frac{p-1}{n-p}
 \left(\frac{\Cap_p(\Omega)}{|\Sph^{n-1}|}\right)^{1/(p-1)}.
\]
Both limits are uniform as $|x|\to\infty$. On $\Sigma_t$, the first relation gives $|x|=(A/t)^{\frac{p-1}{n-p}}(1+o(1))$, while the second relation and the flux identity
$\Cap_p(\Omega)=\int_{\Sigma_t}|\nabla u|^{p-1}dA$ determine the first-order area asymptotic. Consequently,
\begin{equation}\label{eq:direct-asymptotic}
 \lim_{t\downarrow0}\frac1{t^2}
 \left(\frac{\Cap_p(\Omega)}{|\Sigma_t|}\right)^{
 \frac{2(n-p)}{(n-1)(p-1)}}
 =\left(\frac{n-p}{p-1}\right)^2
 \left(\frac{|\Sph^{n-1}|}{\Cap_p(\Omega)}\right)^{2/(n-1)}.
\end{equation}
On the other hand, the monotonicity in \eqref{eq:intro-monotone} and its zero limit at $t=0$ give
\[
 \left(\frac{\Cap_p(\Omega)}{|\Sigma_t|}\right)^{
 \frac{2(n-p)}{(n-1)(p-1)}}
 \ge\frac{n-p}{p-1}\lambda t^2.
\]
Comparing with \eqref{eq:direct-asymptotic}, we find
\[
 \left(\frac{n-p}{p-1}\right)^2
 \left(\frac{|\Sph^{n-1}|}{\Cap_p(\Omega)}\right)^{2/(n-1)}
 \ge\frac{n-p}{p-1}\lambda,
\]
which is exactly the upper bound in \eqref{eq:intro-euclidean-bounds}.

For a ball $B(R)$, the potential is $u(x)=(R/|x|)^{(n-p)/(p-1)}$, so
\[
 |\nabla u|\big|_{\partial B(R)}=\frac{n-p}{(p-1)R},
 \qquad K_{\partial B(R)}=R^{-(n-1)}.
\]
Substitution into \eqref{eq:intro-monotone} and \eqref{eq:intro-euclidean-bounds} shows that all the stated constants are attained by balls.
\end{proof}

We also record the resulting mixed-curvature area inequality.

\begin{remark}\label{rem:mixed-area}
Comparing the lower and upper bounds in \eqref{eq:intro-euclidean-bounds} eliminates the capacity and yields
\begin{equation}\label{eq:mixed-area-consequence}
 |\partial\Omega|\,
 \lambda^{\frac{(n-1)^2}{2(n-p)}}
 \le
 |\mathbb S^{n-1}|
 \left(\frac{n-p}{p-1}\right)^{
 \frac{(n-1)(n+1-2p)}{2(n-p)}}.
\end{equation}
At the critical exponent $p=\frac{n+1}{2}$, this reduces to
\[
 |\partial\Omega|\min_{\partial\Omega}K
 \le |\mathbb S^{n-1}|,
\]
which also follows immediately from the Gauss-map area formula $d\omega=K\,dA$. Indeed, the Gauss map $\nu:\partial\Omega\to\mathbb S^{n-1}$ has degree one for a smooth convex body, so the area formula gives $\int_{\partial\Omega}K\,dA=|\mathbb S^{n-1}|$; see, for example, \cite[Chapter~8]{Schneider}.
\end{remark}

\begin{proof}[Proof of the rigidity statement in Theorem~\ref{thm:mixed-euclidean}]
The converse follows from the ball computation above. Suppose that equality holds in the lower estimate. Since $\Cap_p(\Omega)>0$, equality forces $\lambda>0$. The function in \eqref{eq:intro-monotone} is nondecreasing, tends to zero as $t\downarrow0$, and equals zero at $t=1$. It is therefore identically zero. Hence, for almost every regular value $t$, equality holds in both the arithmetic--geometric mean step and the Jensen step used in the proof.

Equality in
\[
 H_t\ge(n-1)K_t^{1/(n-1)}
\]
forces the shape operator of $\Sigma_t$ to have the form $A_t=\mu_t\Id$. Since $n\ge3$, the hypersurface dimension is at least two; the Codazzi equation then gives $\nabla_{\Sigma_t}\mu_t=0$. Thus $\mu_t$ is constant on each connected regular level, and a closed connected totally umbilical hypersurface in Euclidean space is a round sphere. Equality in Jensen's inequality also forces $|\nabla u|$ to be constant on $\Sigma_t$.

Choose regular values $t_j\uparrow1$ from the full-measure set on which these equalities hold. Then each convex superlevel set $\Omega_{t_j}$ is a Euclidean ball. The sets are nested and, by continuity of $u$, converge in Hausdorff distance to $\overline\Omega$ as $t_j\uparrow1$. A Hausdorff limit of a nested sequence of Euclidean balls with nonempty limiting interior is again a Euclidean ball. Hence $\Omega$ is a ball.
\end{proof}

The same monotonicity also gives pointwise control of the area of every level hypersurface.

\begin{corollary}\label{cor:level-area}
Under the assumptions of Theorem~\ref{thm:mixed-euclidean}, for every $0<t\le1$, with the right-hand side interpreted as $+\infty$ when $\lambda=0$,
\begin{equation}\label{eq:level-area}
\resizebox{0.96\textwidth}{!}{$\displaystyle
 \frac{\Cap_p(\Omega)}{
 \left[
 \left(\frac{\Cap_p(\Omega)}{|\partial\Omega|}\right)^{\frac{2(n-p)}{(n-1)(p-1)}}
 -\frac{n-p}{p-1}\lambda(1-t^2)
 \right]^{\frac{(n-1)(p-1)}{2(n-p)}}}
 \le |\Sigma_t|\le
 \frac{\Cap_p(\Omega)}{
 \left(\frac{n-p}{p-1}\lambda\right)^{\frac{(n-1)(p-1)}{2(n-p)}}
 t^{\frac{(n-1)(p-1)}{n-p}}}. $}
\end{equation}
\end{corollary}

\begin{proof}
Applying the monotonicity in \eqref{eq:intro-monotone} on $[0,t]$ and $[t,1]$ gives
\[
 \left(\frac{\Cap_p(\Omega)}{|\Sigma_t|}\right)^{\frac{2(n-p)}{(n-1)(p-1)}}
 \ge\frac{n-p}{p-1}\lambda t^2,
\]
and
\[
 \left(\frac{\Cap_p(\Omega)}{|\Sigma_t|}\right)^{\frac{2(n-p)}{(n-1)(p-1)}}
 \le
 \left(\frac{\Cap_p(\Omega)}{|\partial\Omega|}\right)^{\frac{2(n-p)}{(n-1)(p-1)}}
 -\frac{n-p}{p-1}\lambda(1-t^2).
\]
Solving for $|\Sigma_t|$ proves \eqref{eq:level-area}.
\end{proof}

The critical exponent $p=\frac{n+1}{2}$ is distinguished by the fact that $n+1-2p=0$.

\begin{remark}\label{cor:critical}
At $p=\frac{n+1}{2}$, the quantity in \eqref{eq:intro-lambda} becomes $(\min_{\partial\Omega}K)^{1/(n-1)}$. Thus, if
\[
 K\ge K_0>0\qquad\text{on }\partial\Omega,
\]
then $\lambda\ge K_0^{1/(n-1)}$, and \eqref{eq:intro-euclidean-bounds} gives
\begin{equation}\label{eq:critical-bounds}
 |\partial\Omega|K_0^{1/2}
 \le\Cap_{\frac{n+1}{2}}(\Omega)
 \le|\Sph^{n-1}|K_0^{-1/2}.
\end{equation}
For $n=3$, $p=2$, and $K_0=1$, this reads $|\partial\Omega|\le\Cap_2(\Omega)\le4\pi$.
\end{remark}

\section{A Euclidean curvature-integral lower bound}\label{sec:euclidean-ray}

We first apply the normal-ray mechanism to a smooth bounded convex domain $\Omega\subset\mathbb R^n$ with nonnegative principal curvatures. The outward normal map
\[
 (y,s)\longmapsto y+s\nu(y),
 \qquad y\in\partial\Omega,\quad s\ge0,
\]
is a global parametrization of $\mathbb R^n\setminus\Omega$, and its Jacobian is
\begin{equation}\label{eq:euclidean-normal-jacobian}
 J(y,s)=\prod_{i=1}^{n-1}(1+s\kappa_i(y)).
\end{equation}

The following theorem is the Euclidean normal-ray curvature lower bound.

\begin{theorem}\label{thm:euclidean-ray}
Let $\Omega\subset\mathbb R^n$ be smooth, bounded, and strictly convex, let $1<p<n$, and let $\kappa_1,\ldots,\kappa_{n-1}$ be the principal curvatures of $\partial\Omega$. Then
\begin{align}
 \Cap_p(\Omega)
 &\ge
 \int_{\partial\Omega}
 \left[
 \int_0^\infty
 \prod_{i=1}^{n-1}(1+s\kappa_i)^{-1/(p-1)}\,ds
 \right]^{1-p}dA
 \label{eq:intro-euclidean-ray}\\
 &\ge
 \left(\frac{n-p}{p-1}\right)^{p-1}
 \int_{\partial\Omega}K^{\frac{p-1}{n-1}}\,dA.
 \label{eq:intro-euclidean-curvature}
\end{align}
Here, if the inner ray integral is infinite, its power $1-p$ is understood as zero. Equality in either inequality holds if and only if $\Omega$ is a Euclidean ball.
\end{theorem}

\begin{proof}[Proof of Theorem~\ref{thm:euclidean-ray}]
Let $\varphi\in C_c^\infty(\mathbb R^n)$ be admissible for $\Cap_p(\Omega)$. In normal coordinates,
\[
 \int_{\mathbb R^n\setminus\Omega}|\nabla\varphi|^p\,dx
 \ge
 \int_{\partial\Omega}\int_0^\infty
 |\partial_s\varphi(y,s)|^pJ(y,s)\,ds\,dA(y).
\]
Applying Lemma~\ref{lem:weighted-ray} along every normal ray and then using \eqref{eq:euclidean-normal-jacobian}, we obtain
\[
 \int_{\mathbb R^n\setminus\Omega}|\nabla\varphi|^p\,dx
 \ge
 \int_{\partial\Omega}
 \left[
 \int_0^\infty
 \prod_{i=1}^{n-1}(1+s\kappa_i)^{-1/(p-1)}\,ds
 \right]^{1-p}dA.
\]
Taking the infimum over admissible $\varphi$ proves \eqref{eq:intro-euclidean-ray}.

For the second inequality, Maclaurin's inequalities give
\[
 \prod_{i=1}^{n-1}(1+s\kappa_i)
 \ge
 \left(1+sK^{1/(n-1)}\right)^{n-1}.
\]
Consequently,
\[
\begin{aligned}
 \int_0^\infty
 \prod_{i=1}^{n-1}(1+s\kappa_i)^{-1/(p-1)}\,ds
 &\le
 \int_0^\infty
 \left(1+sK^{1/(n-1)}\right)^{-\frac{n-1}{p-1}}\,ds\\
 &=\frac{p-1}{n-p}K^{-1/(n-1)}.
\end{aligned}
\]
Because $1-p<0$, substitution into the first inequality gives \eqref{eq:intro-euclidean-curvature}.
\end{proof}

\begin{proof}[Proof of the rigidity statement in Theorem~\ref{thm:euclidean-ray}]
Balls give equality. Conversely, let $u$ be the exterior $p$-capacitary potential and assume equality in \eqref{eq:intro-euclidean-ray}; equality in \eqref{eq:intro-euclidean-curvature} implies equality in \eqref{eq:intro-euclidean-ray} as well, since the proof gives a chain of lower bounds. Because $u=1$ on $\partial\Omega$, $u\to0$ at infinity, and $u$ has finite $p$-energy, Lemma~\ref{lem:weighted-ray} applies directly to the restriction $s\mapsto u(y,s)$ on every outward normal ray. Hence
\[
\begin{aligned}
 \Cap_p(\Omega)
 &=\int_{\mathbb R^n\setminus\Omega}|\nabla u|^p\,dx\\
 &\ge
 \int_{\partial\Omega}\int_0^\infty
 |\partial_su|^pJ\,ds\,dA\\
 &\ge
 \int_{\partial\Omega}
 \left[
 \int_0^\infty J(y,s)^{-1/(p-1)}\,ds
 \right]^{1-p}dA.
\end{aligned}
\]
Equality in \eqref{eq:intro-euclidean-ray} forces equality in both inequalities above. In particular, equality in $|\nabla u|^p\ge|\partial_su|^p$ implies that the tangential gradient of $u$ vanishes almost everywhere; by regularity it vanishes everywhere. Thus, in the global normal coordinates, $u(y,s)=f(s)$. The convex-ring theory of Lewis \cite{Lewis} gives $|\nabla u|>0$ in the exterior, so $f'(s)<0$. The $p$-harmonic equation becomes
\[
 (p-1)f''(s)+f'(s)H_s(y)=0,
\]
where $H_s$ is the mean curvature of the parallel hypersurface at distance $s$. Since $f'(s)\ne0$, $H_s$ is independent of $y$ for every $s$, and in particular $\partial\Omega$ has constant mean curvature. By Alexandrov's theorem \cite{Alexandrov} a closed embedded constant-mean-curvature hypersurface in Euclidean space is a round sphere. Hence $\Omega$ is a ball.
\end{proof}

\subsection{Gauss-curvature comparison and pinching}

\begin{proof}[Proof of Theorem~\ref{thm:intro-gauss-comparison}]
The lower comparison follows directly from Theorem~\ref{thm:euclidean-ray}. Indeed, for every $1<p<n$,
\[
\begin{aligned}
 \Cap_p(\Omega)
 &\ge
 \left(\frac{n-p}{p-1}\right)^{p-1}
 \int_{\partial\Omega}K^{\frac{p-1}{n-1}}\,dA\\
 &=
 \left(\frac{n-p}{p-1}\right)^{p-1}
 \int_{\partial\Omega}K\,K^{-\frac{n-p}{n-1}}\,dA\\
 &\ge
 |\mathbb S^{n-1}|
 \left(\frac{n-p}{p-1}\right)^{p-1}
 K_+^{-\frac{n-p}{n-1}}\\
 &=\Cap_p\!\left(B(K_+^{\frac{1}{1-n}})\right),
\end{aligned}
\]
where we used the Gauss-map identity $\int_{\partial\Omega}K\,dA=|\mathbb S^{n-1}|$.

For the upper comparison, assume $1<p\le\frac{n+1}{2}$ and let $u$ be the exterior $\frac{n+1}{2}$-capacitary potential of $\Omega$. For $\Sigma_t=\{u=t\}$, Theorem~\ref{thm:mixed-euclidean} and Corollary~\ref{cor:level-area}, specialized to the critical exponent, give
\[
 \Cap_{\frac{n+1}{2}}(\Omega)\le |\mathbb S^{n-1}|K_-^{-1/2},
 \qquad
 |\Sigma_t|\le \Cap_{\frac{n+1}{2}}(\Omega)K_-^{-1/2}t^{-(n-1)}.
\]
The critical flux identity and H\"older's inequality therefore yield
\[
\begin{aligned}
 \int_{\Sigma_t}|\nabla u|^{p-1}\,dA
 &\le
 \left(\int_{\Sigma_t}|\nabla u|^{\frac{n-1}{2}}\,dA\right)^{\frac{2(p-1)}{n-1}}
 |\Sigma_t|^{\frac{n+1-2p}{n-1}}\\
 &\le
 \Cap_{\frac{n+1}{2}}(\Omega)
 K_-^{-\frac{n+1-2p}{2(n-1)}}
 t^{-(n+1-2p)}.
\end{aligned}
\]
After the standard cutoff at infinity, $u^{(n-p)/(p-1)}$ is an admissible test function for $\Cap_p(\Omega)$. By the coarea formula,
\[
\begin{aligned}
 \Cap_p(\Omega)
 &\le
 \left(\frac{n-p}{p-1}\right)^p
 \int_0^1
 t^{\frac{p(n+1-2p)}{p-1}}
 \int_{\Sigma_t}|\nabla u|^{p-1}\,dA\,dt\\
 &\le
 \left(\frac{n-p}{p-1}\right)^{p-1}
 \Cap_{\frac{n+1}{2}}(\Omega)
 K_-^{-\frac{n+1-2p}{2(n-1)}}\\
 &\le
 |\mathbb S^{n-1}|
 \left(\frac{n-p}{p-1}\right)^{p-1}
 K_-^{-\frac{n-p}{n-1}}\\
 &=\Cap_p\!\left(B(K_-^{\frac{1}{1-n}})\right).
\end{aligned}
\]
This proves \eqref{eq:intro-gauss-comparison}.

If equality holds in the left-hand comparison, then equality must hold in the chain above, in particular in the normal-ray curvature estimate of Theorem~\ref{thm:euclidean-ray}; its rigidity statement therefore implies that $\Omega$ is a Euclidean ball. Conversely, balls give equality.

Assume now $1<p<\frac{n+1}{2}$ and equality holds in the
right-hand comparison. Then equality holds in every step of the
upper-bound chain. In particular, equality holds in H\"older's
inequality for almost every regular value $t$, so $|\nabla u|$ is
constant on $\Sigma_t$ for almost every $t\in(0,1)$. Since the
critical capacitary potential has no critical points in the exterior \cite{Lewis},
its level hypersurfaces form a smooth foliation; by continuity, the
same constancy holds for every $t\in(0,1)$. Hence there exists a
smooth function
\[
G:(0,1)\longrightarrow(0,\infty)
\]
such that
\[
|\nabla u|=G(u)
\qquad\text{in }\mathbb R^n\setminus\overline\Omega.
\]

Recall that
$
\nu=-\frac{\nabla u}{|\nabla u|},
$
so that $\nabla u=-G(u)\nu$. Since $u$ is
$\frac{n+1}{2}$-harmonic, on $\Sigma_t$ we have
\[
0
=\operatorname{div}
\left(|\nabla u|^{\frac{n-3}{2}}\nabla u\right) =\operatorname{div}
\left(-G(u)^{\frac{n-1}{2}}\nu\right) =G(t)^{\frac{n-1}{2}}
\left(\frac{n-1}{2}G'(t)-H_t\right).
\]
Therefore
\[
H_t=\frac{n-1}{2}G'(t)
\qquad\text{on }\Sigma_t.
\]
Thus $H_t$ is constant on each level hypersurface $\Sigma_t$.
By Alexandrov's theorem \cite{Alexandrov}, every $\Sigma_t$ is a Euclidean sphere.
Letting $t\uparrow1$ shows that $\Omega$ itself is a Euclidean ball.
The converse follows immediately from the ball formula
\eqref{eq:intro-ball-capacity}.

It remains to prove rigidity when
\[
 p=\frac{n+1}{2}.
\]
Assume that equality holds in the right-hand comparison, namely
\begin{equation}\label{eq:endpoint-capacity-equality}
 \Cap_{\frac{n+1}{2}}(\Omega)
 =
 |\mathbb S^{n-1}|K_-^{-1/2}.
\end{equation}
Let $u$ be the exterior $\frac{n+1}{2}$-capacitary potential of
$\Omega$, and set
\[
 m(t):=\min_{\Sigma_t}K_t^{1/(n-1)},
 \qquad 0<t\le1.
\]
For every $\varepsilon\in(0,1)$, apply the height-concavity theorem of
Ma--Zhang \cite{MaZhang} to
\[
 v_\varepsilon=\frac{u-\varepsilon}{1-\varepsilon}
\]
on the convex ring bounded by $\partial\Omega$ and
$\Sigma_\varepsilon$. Since
\[
 n+1-2p=0
\]
for $p=(n+1)/2$, the mixed curvature quantity in their theorem reduces
to $K^{1/(n-1)}$. Hence $m$ is concave on every interval
$[\varepsilon,1]$, and therefore on $(0,1]$.

The Gauss-map identity on $\Sigma_t$ gives
\[
 |\mathbb S^{n-1}|
 =
 \int_{\Sigma_t}K_t\,dA
 \ge
 m(t)^{n-1}|\Sigma_t|,
\]
and thus
\begin{equation}\label{eq:endpoint-m-area}
 m(t)
 \le
 \left(
 \frac{|\mathbb S^{n-1}|}{|\Sigma_t|}
 \right)^{1/(n-1)}.
\end{equation}
Specializing the far-field area asymptotic
\eqref{eq:direct-asymptotic} to $p=(n+1)/2$, and then using
\eqref{eq:endpoint-capacity-equality}, yields
\begin{equation}\label{eq:endpoint-area-asymptotic}
 \lim_{t\downarrow0}
 \frac1t
 \left(
 \frac{|\mathbb S^{n-1}|}{|\Sigma_t|}
 \right)^{1/(n-1)}
 =
 \left(
 \frac{|\mathbb S^{n-1}|}
 {\Cap_{\frac{n+1}{2}}(\Omega)}
 \right)^{2/(n-1)}
 =
 K_-^{1/(n-1)}.
\end{equation}
It follows from \eqref{eq:endpoint-m-area} that $m(t)\to0$ and
\[
 \limsup_{t\downarrow0}\frac{m(t)}t
 \le K_-^{1/(n-1)}.
\]
Define $m(0)=0$. Since $m$ is concave on every $[\varepsilon,1]$ and
has a finite limit at $0$, it extends to a concave function on
$[0,1]$. Moreover,
\[
 m(1)=K_-^{1/(n-1)}.
\]
Concavity therefore gives
\[
 m(t)\ge tK_-^{1/(n-1)}.
\]
Together with the preceding $\limsup$ estimate this implies
\[
 \lim_{t\downarrow0}\frac{m(t)}t=K_-^{1/(n-1)}.
\]
For a concave function vanishing at the origin, the quotient
$m(t)/t$ is nonincreasing. Consequently,
\begin{equation}\label{eq:endpoint-m-affine}
 m(t)=tK_-^{1/(n-1)},
 \qquad 0\le t\le1.
\end{equation}

We now return to the differential inequality underlying the
Ma--Zhang concavity theorem. Parametrize the strictly convex level
hypersurfaces by their Gauss maps and write
\[
 W(\theta,t):=K_t(\theta)^{1/(n-1)},
 \qquad
 (\theta,t)\in\mathbb S^{n-1}\times(0,1).
\]
In the support-function formulation of Ma--Zhang, the $p$-Laplace
equation becomes the nonlinear equation \cite[(2.7)]{MaZhang}.
The operator $\mathcal L$ introduced in \cite[(2.8)]{MaZhang} is
precisely the linearized operator of that equation; in particular its
principal part is elliptic. Their proof of \cite[(3.1)]{MaZhang},
with
\[
 \alpha=n+1-2p=0,
 \qquad
 \beta=-\frac1{n-1},
\]
gives
\begin{equation}\label{eq:endpoint-MZ-supersolution}
 \mathcal LW\le0
 \quad\bmod\ \nabla_\theta(-\log K_t).
\end{equation}
Here ``modulo'' has the same meaning as in Ma--Zhang: the omitted
terms are linear in the tangential derivatives of $-\log K_t$, with
locally bounded coefficients. Equivalently, after moving those terms
to the left-hand side and using
\[
 \nabla_\theta(-\log K_t)
 =
 -(n-1)\frac{\nabla_\theta W}{W},
\]
one obtains, locally in
$\mathbb S^{n-1}\times(0,1)$, an operator
\begin{equation}\label{eq:endpoint-drift-operator}
 \widetilde{\mathcal L}
 =
 \mathcal L+B^i(\theta,t)\partial_{\theta_i}
\end{equation}
such that
\begin{equation}\label{eq:endpoint-drift-supersolution}
 \widetilde{\mathcal L}W\le0.
\end{equation}

For completeness, fix $0<a<b<1$ and put
\[
 Q_{a,b}:=\mathbb S^{n-1}\times(a,b).
\]
The cylinder $Q_{a,b}$ is open, not compact. What is compact is its
closure
\[
 \overline Q_{a,b}
 =
 \mathbb S^{n-1}\times[a,b].
\]
Because the capacitary potential has no critical points and its level
hypersurfaces are smooth and strictly convex, the support function and
all coefficients occurring in $\mathcal L$ are smooth on a
neighborhood of $\overline Q_{a,b}$. Hence the coefficients of
$\widetilde{\mathcal L}$ are bounded on $\overline Q_{a,b}$, while
the positive-definite principal coefficient matrix has a positive
minimum eigenvalue there. Thus $\widetilde{\mathcal L}$ is uniformly
elliptic on the connected open cylinder $Q_{a,b}$, which is exactly
the setting in which the strong minimum principle is applied.

By \cite[(2.8)]{MaZhang}, $\mathcal L$ contains only second-order
derivatives, while the additional drift in
\eqref{eq:endpoint-drift-operator} is purely tangential. Therefore
\[
 \widetilde{\mathcal L}t=0.
\]
Set
\[
 Z(\theta,t)
 :=
 W(\theta,t)-tK_-^{1/(n-1)}.
\]
By \eqref{eq:endpoint-m-affine}, $Z\ge0$, and for each fixed
$t\in(0,1)$ the function $Z(\cdot,t)$ vanishes at a point where
$W(\cdot,t)$ attains its minimum. Moreover,
\[
 \widetilde{\mathcal L}Z\le0.
\]
Choose $0<a<t_0<b<1$ and $\theta_0\in\mathbb S^{n-1}$ so that
$Z(\theta_0,t_0)=0$. Then $Z$ is a nonnegative supersolution attaining
its minimum at an interior point of the connected open set
$Q_{a,b}$. The strong minimum principle therefore gives
$Z\equiv0$ on $Q_{a,b}$. Since $a,b$ are arbitrary,
\begin{equation}\label{eq:endpoint-K-pointwise}
 K_t(\theta)^{1/(n-1)}
 =
 tK_-^{1/(n-1)}
 \qquad
 \text{on }\mathbb S^{n-1}\times(0,1).
\end{equation}
Letting $t\uparrow1$ yields
\[
 K\equiv K_-
 \qquad\text{on }\partial\Omega.
\]
By the uniqueness theorem for the Minkowski problem, a closed smooth
strictly convex hypersurface with constant Gauss--Kronecker curvature
is a round sphere. Hence $\Omega$ is a Euclidean ball.
\end{proof}

We next use rotational ellipsoids to exhibit both the failure of the corresponding ball inclusions and the supercritical obstruction.

\begin{example}\label{ex:ellipsoid-obstructions}
Assume $n\ge4$ and consider
\[
 E_t=\left\{x_1^2+\cdots+x_{n-1}^2+\frac{x_n^2}{t^2}\le1\right\},
 \qquad t>0,
\]
with parametrization
\[
 X(\varphi,\omega)=(\sin\varphi\,\omega,t\cos\varphi),
 \qquad \omega\in\mathbb S^{n-2}.
\]
Its Gauss--Kronecker curvature is
\begin{equation}\label{eq:ellipsoid-gauss-curvature}
 K(\varphi)
 =\frac{t^{n-1}}
 {\bigl(\cos^2\varphi+t^2\sin^2\varphi\bigr)^{\frac{n+1}{2}}}.
\end{equation}
The following three facts show both that the comparisons in \eqref{eq:intro-gauss-comparison} are not induced by ball inclusions and that the range of the right-hand comparison cannot in general be extended beyond $\frac{n+1}{2}$.
\begin{enumerate}
\item If $t>1$, then $K_-=t^{-2}$ and no ball of radius $K_-^{\frac{1}{1-n}}$ contains $E_t$.
\item If $0<t<1$, then $K_+=t^{-2}$ and no ball of radius $K_+^{\frac{1}{1-n}}$ is contained in $E_t$.
\item If $\frac{n+1}{2}<p<n$, then
\begin{equation}\label{eq:ellipsoid-supercritical-failure}
 \lim_{t\to\infty}
 \frac{\Cap_p(E_t)}
 {\Cap_p\!\left(B(K_-^{\frac{1}{1-n}})\right)}
 =+\infty.
\end{equation}
In particular, for all sufficiently large $t$, the right-hand inequality in \eqref{eq:intro-gauss-comparison} fails.
\end{enumerate}

\begin{proof}
If $t>1$, \eqref{eq:ellipsoid-gauss-curvature} gives
\[
 K_-=t^{-2},\qquad K_+=t^{n-1}.
\]
The diameter of $E_t$ is $2t$, so every ball containing $E_t$ has radius at least $t$. Since $n\ge4$,
\[
 K_-^{\frac{1}{1-n}}=t^{2/(n-1)}<t,
\]
which proves (1).

If $0<t<1$, \eqref{eq:ellipsoid-gauss-curvature} instead gives
\[
 K_-=t^{n-1},\qquad K_+=t^{-2}.
\]
The inradius of $E_t$ is $t$, while
\[
 K_+^{\frac{1}{1-n}}=t^{2/(n-1)}>t
\]
for $n\ge4$. This proves (2).

For (3), first suppose
\[
 \frac{n+1}{2}<p<n-1.
\]
For $t\ge1$, the ellipsoid contains the cylinder
\[
 C_t:=B^{n-1}\!\left(\frac1{\sqrt2}\right)
 \times\left[-\frac{t}{\sqrt2},\frac{t}{\sqrt2}\right],
\]
where $B^{n-1}(r)$ denotes the Euclidean ball of radius $r$ in $\mathbb R^{n-1}$. Let $\varphi\in C_c^\infty(\mathbb R^n)$ be admissible for $\Cap_p(C_t)$ and write $x=(x',x_n)$. For every $x_n\in[-t/\sqrt2,t/\sqrt2]$, the function $x'\mapsto\varphi(x',x_n)$ is admissible for the $p$-capacity of $B^{n-1}(1/\sqrt2)$ in $\mathbb R^{n-1}$. Since $p<n-1$, this latter capacity is positive. Hence, by Fubini,
\[
\begin{aligned}
 \int_{\mathbb R^n}|\nabla\varphi|^p\,dx
 &\ge
 \int_{-t/\sqrt2}^{t/\sqrt2}
 \int_{\mathbb R^{n-1}}|\nabla_{x'}\varphi|^p\,dx'\,dx_n\\
 &\ge c_{n,p}\,t.
\end{aligned}
\]
Taking the infimum and using capacity monotonicity gives
\[
 \Cap_p(E_t)\ge c_{n,p}\,t.
\]
Since $K_-=t^{-2}$,
\[
 \Cap_p\!\left(B(K_-^{\frac{1}{1-n}})\right)
 =|\mathbb S^{n-1}|
 \left(\frac{n-p}{p-1}\right)^{p-1}
 t^{\frac{2(n-p)}{n-1}},
\]
and therefore
\[
 \frac{\Cap_p(E_t)}
 {\Cap_p\!\left(B(K_-^{\frac{1}{1-n}})\right)}
 \ge c_{n,p}\,
 t^{\frac{2p-n-1}{n-1}}\longrightarrow+\infty.
\]

It remains to consider $n-1\le p<n$. Fix any
$
 p_0\in\left(\frac{n+1}{2},n-1\right).
$
The estimate just proved, applied with $p_0$, gives
\[
 \Cap_{p_0}(E_t)\ge c_{n,p_0}\,t.
\]
By the monotonicity in Lemma~\ref{lem:p-capacity-exponent-monotonicity} below,
\[
 \Cap_p(E_t)\ge c_{n,p,p_0}\,
 t^{\frac{n-p}{n-p_0}}.
\]
Consequently,
\[
 \frac{\Cap_p(E_t)}
 {\Cap_p\!\left(B(K_-^{\frac{1}{1-n}})\right)}
 \ge c_{n,p,p_0}\,
 t^{(n-p)\left(\frac1{n-p_0}-\frac2{n-1}\right)}
 \longrightarrow+\infty,
\]
because $p_0>\frac{n+1}{2}$. This proves (3).
\end{proof}
\end{example}

The following normalized monotonicity is related to Xu's quasi-monotonicity for variational $p$-capacity \cite{Xu1996} and to Xiao's affine counterparts \cite{XiaoAffineCapacity,XiaoAffineRedux}.

\begin{lemma}\label{lem:p-capacity-exponent-monotonicity}
Let $E\subset\mathbb R^n$ be compact. Then
\[
 p\longmapsto
 \left[
 \left(\frac{(n-1)p}{n(n-p)|B(1)|^{1/n}}\right)^p
 \Cap_p(E)
 \right]^{\frac1{n-p}}
\]
is nondecreasing on $(1,n)$.
\end{lemma}

\begin{proof}
Fix $1<p<q<n$. We use the standard equivalent Sobolev formulation of capacity. The sharp $W^{1,1}(\mathbb R^n)$ Sobolev inequality is
\[
 \|f\|_{\frac n{n-1}}
 \le \frac1{n|B(1)|^{1/n}}\|\nabla f\|_1.
\]
Applying it to $f=u^{\frac{q(n-1)}{n-q}}$ and then using H\"older's inequality gives
\[
 \|u\|_{\frac{nq}{n-q}}
 \le
 \frac{(n-1)q}{n(n-q)|B(1)|^{1/n}}
 \|\nabla u\|_q.
\]
Indeed, after differentiating the power of $u$, the second H\"older factor has exponent $nq/(n-q)$, and cancellation leaves exactly one power of the norm on the left.

Now let $u$ be any nonnegative admissible function for $\Cap_q(E)$. Then
$
 u^{\frac{q(n-p)}{p(n-q)}}
$
is admissible for $\Cap_p(E)$. Hence H\"older's inequality and the preceding Sobolev estimate yield
\begin{align*}
 \Cap_p(E)
 &\le
 \left(\frac{q(n-p)}{p(n-q)}\right)^p
 \int_{\mathbb R^n}
 u^{\frac{n(q-p)}{n-q}}|\nabla u|^p\,dx\\
 &\le
 \left(\frac{q(n-p)}{p(n-q)}\right)^p
 \left(\int_{\mathbb R^n}|\nabla u|^q\,dx\right)^{p/q}
 \left(\int_{\mathbb R^n}u^{\frac{nq}{n-q}}\,dx\right)^{(q-p)/q}\\
 &\le
 \left(\frac{q(n-p)}{p(n-q)}\right)^p
 \left(\frac{(n-1)q}{n(n-q)|B(1)|^{1/n}}\right)^{\frac{n(q-p)}{n-q}}
 \left(\int_{\mathbb R^n}|\nabla u|^q\,dx\right)^{\frac{n-p}{n-q}}.
\end{align*}
Taking the infimum over all such $u$ and rearranging the resulting inequality proves the claimed monotonicity.
\end{proof}

\begin{remark}\label{rem:cross-exponent-curvature}
For $1<r<n$, set
\[
 a_{n,r}:=
 \left(\frac{(n-1)r}{n(n-r)|B(1)|^{1/n}}\right)^r.
\]
Then Lemma~\ref{lem:p-capacity-exponent-monotonicity} says that, for
$1<s<q<n$,
\[
 \left[a_{n,q}\Cap_q(\Omega)\right]^{\frac1{n-q}}
 \ge
 \left[a_{n,s}\Cap_s(\Omega)\right]^{\frac1{n-s}}.
\]
Combining this with Theorem~\ref{thm:euclidean-ray}, applied with exponent $s$, gives
\begin{equation}\label{eq:optimized-cross-exponent-curvature}
 \Cap_q(\Omega)
 \ge
 \sup_{1<s\le q}
 a_{n,q}^{-1}
 \left[
 a_{n,s}
 \left(\frac{n-s}{s-1}\right)^{s-1}
 \int_{\partial\Omega}K^{\frac{s-1}{n-1}}\,dA
 \right]^{\frac{n-q}{n-s}}.
\end{equation}
Thus the $q$-capacity can be controlled by any lower Gauss-curvature moment with $s<q$, while the choice $s=q$ recovers the direct estimate.
\end{remark}

\begin{remark}\label{rem:eigenvalue-curvature}
Let
\[
 \lambda_{1,p}(\Omega)
 :=\inf_{0\ne v\in W_0^{1,p}(\Omega)}
 \frac{\displaystyle\int_\Omega |\nabla v|^p\,dx}
 {\displaystyle\int_\Omega |v|^p\,dx}
\]
be the first Dirichlet eigenvalue of the $p$-Laplacian. If $K_->0$, then for every $p>1$,
\begin{equation}\label{eq:eigenvalue-kmin}
 \lambda_{1,p}(\Omega)
 \ge \lambda_{1,p}(B(1))\,
 K_-^{\frac{p}{n-1}}.
\end{equation}
Indeed, the Gauss-map identity and $K\ge K_-$ give
$
 |\partial\Omega|\le \frac{|\mathbb S^{n-1}|}{K_-},
$
and the Euclidean isoperimetric inequality therefore yields
$
 |\Omega|\le |B(1)|K_-^{-\frac{n}{n-1}}.
$
The $p$-Faber--Krahn inequality \cite{Bhattacharya} then gives \eqref{eq:eigenvalue-kmin}. Equality holds if and only if $\Omega$ is a Euclidean ball.

In contrast, when $n\ge4$, $K_+$ cannot give a universal upper bound for the first Dirichlet $p$-eigenvalue. More precisely, there is no constant $C_{n,p}$ such that
\[
 \lambda_{1,p}(\Omega)
 \le C_{n,p}K_+^{\frac{p}{n-1}}
\]
for every smooth bounded strictly convex domain. Indeed, take the ellipsoids $E_t$ from Example~\ref{ex:ellipsoid-obstructions} with $0<t<1$. Then
\[
 K_+(E_t)=t^{-2},
 \qquad
 E_t\subset\{x\in\mathbb R^n:\ |x_n|<t\}.
\]
Applying the Dirichlet Poincar\'e inequality in the $x_n$-direction gives
$
 \lambda_{1,p}(E_t)
 \ge \lambda_{1,p}((-1,1))\,t^{-p}.
$
Consequently,
\[
 \frac{\lambda_{1,p}(E_t)}{K_+(E_t)^{\frac{p}{n-1}}}
 \ge
 \lambda_{1,p}((-1,1))\,
 t^{-\frac{p(n-3)}{n-1}}
 \longrightarrow+\infty
 \qquad\text{as }t\downarrow0.
\]
Thus $K_-$ provides the sharp scale-invariant lower bound \eqref{eq:eigenvalue-kmin}, whereas, in dimensions $n\ge4$, $K_+$ alone cannot control $\lambda_{1,p}$ from above.
\end{remark}

\begin{proof}[Proof of Theorem~\ref{thm:intro-gauss-pinching}]
The case $\rho=1$ follows from the limiting form of the argument below, so assume $0<\rho<1$. By Theorem~\ref{thm:euclidean-ray},
\[
 \Cap_p(\Omega)
 \ge
 \left(\frac{n-p}{p-1}\right)^{p-1}
 \int_{\partial\Omega}K^{\frac{p-1}{n-1}}\,dA.
\]
Since $K_-\ge\rho K_+$ and $x^{\frac{p-1}{n-1}}$ is concave on $(0,\infty)$, its graph lies above the chord joining $\rho K_+$ and $K_+$. Therefore, pointwise on $\partial\Omega$,
\[
\begin{aligned}
 K^{\frac{p-1}{n-1}}
 \ge{}&
 \frac{(K_+-K)(\rho K_+)^{\frac{p-1}{n-1}}
 +(K-\rho K_+)K_+^{\frac{p-1}{n-1}}}
 {(1-\rho)K_+}\\
 ={}&
 \frac{K_+^{\frac{p-1}{n-1}-1}}{1-\rho}
 \left[
 \left(\rho^{\frac{p-1}{n-1}}-\rho\right)K_+
 +\left(1-\rho^{\frac{p-1}{n-1}}\right)K
 \right].
\end{aligned}
\]
Integrating and using the Gauss-map area formula
\[
 \int_{\partial\Omega}K\,dA=|\mathbb S^{n-1}|
\]
gives
$$
 \int_{\partial\Omega}K^{\frac{p-1}{n-1}}\,dA
 \ge
 \frac{K_+^{\frac{p-1}{n-1}-1}}{1-\rho}
 \bigg[
 \left(\rho^{\frac{p-1}{n-1}}-\rho\right)K_+|\partial\Omega|
 +\left(1-\rho^{\frac{p-1}{n-1}}\right)|\mathbb S^{n-1}|
 \bigg].$$
After division by $|\partial\Omega|^{\frac{n-p}{n-1}}$, this becomes
\[
\begin{aligned}
 \frac{\displaystyle\int_{\partial\Omega}K^{\frac{p-1}{n-1}}\,dA}
 {|\partial\Omega|^{\frac{n-p}{n-1}}}
 &\ge{}\frac{1}{1-\rho}
 \bigg[
 \left(\rho^{\frac{p-1}{n-1}}-\rho\right)
 \left(K_+|\partial\Omega|\right)^{\frac{p-1}{n-1}}\\
 &\qquad+
 \left(1-\rho^{\frac{p-1}{n-1}}\right)|\mathbb S^{n-1}|
 \left(K_+|\partial\Omega|\right)^{-\frac{n-p}{n-1}}
 \bigg]
 \\ &\ge|\mathbb S^{n-1}|^{\frac{p-1}{n-1}}
 \frac{\left(1-\rho^{\frac{p-1}{n-1}}\right)^{\frac{p-1}{n-1}}}
 {(1-\rho)\left(\frac{p-1}{n-1}\right)^{\frac{p-1}{n-1}}}\\
 &\qquad\times
 \frac{\left(\rho^{\frac{p-1}{n-1}}-\rho\right)^{\frac{n-p}{n-1}}}
 {\left(\frac{n-p}{n-1}\right)^{\frac{n-p}{n-1}}}.
\end{aligned}
\]
Here in the second inequality we use the weighted arithmetic--geometric mean inequality with weights $(n-p)/(n-1)$ and $(p-1)/(n-1).$
Combining this with Theorem~\ref{thm:euclidean-ray}, using
\[
 \rho^{\frac{p-1}{n-1}}-\rho
 =\rho^{\frac{p-1}{n-1}}
 \left(1-\rho^{\frac{n-p}{n-1}}\right),
\]
raising the resulting inequality to the power $n-1$, and collecting all powers of
$p-1$ and $n-p$ on the right-hand side gives
\[
\begin{aligned}
 \frac{\Cap_p(\Omega)^{n-1}}{|\partial\Omega|^{n-p}}
 \ge{}& |\mathbb S^{n-1}|^{p-1}
 \frac{(n-1)^{n-1}(n-p)^{n(p-2)+1}}
 {(p-1)^{n(p-1)}}\\
 &\quad\times
 \frac{\rho^{\frac{(p-1)(n-p)}{n-1}}
 \left(1-\rho^{\frac{p-1}{n-1}}\right)^{p-1}}
 {(1-\rho)^{n-1}}\\
 &\quad\times
 \left(1-\rho^{\frac{n-p}{n-1}}\right)^{n-p},
\end{aligned}
\]
which is \eqref{eq:intro-gauss-pinching}.

For $n=3$ and $p=2$, \eqref{eq:intro-gauss-pinching} reduces to
\[
 \frac{\Cap_2(\Omega)^2}{4\pi|\partial\Omega|}
 \ge\frac{4\sqrt\rho}{(1+\sqrt\rho)^2},
\]
which is equivalent to \eqref{eq:intro-newtonian-pinching}. The right-hand side of \eqref{eq:intro-newtonian-pinching} is at least $\sqrt{32/\pi}$ precisely when
\[
 \rho\ge
 \left(\frac{\pi^2-4-\pi\sqrt{\pi^2-8}}{4}\right)^2.
\]
Under this condition,
\[
 \frac{\Cap_2(\Omega)^2}{|\partial\Omega|}\ge\frac{32}{\pi},
\]
which is the P\'olya--Szeg\"{o} inequality in the capacity normalization used here.
\end{proof}

In dimension three and for $p=2$, the Euclidean ray integral has an explicit logarithmic-mean form.

\begin{remark}\label{rem:logarithmic-curvature}
When $n=3$ and $p=2$, the inner one-dimensional integral can be evaluated explicitly. At a point where $\kappa_1,\kappa_2>0$ and $\kappa_1\ne\kappa_2$,
\[
 \left[
 \int_0^\infty
 \frac{ds}{(1+s\kappa_1)(1+s\kappa_2)}
 \right]^{-1}
 =\frac{\kappa_1-\kappa_2}{\ln\kappa_1-\ln\kappa_2}.
\]
At a positive umbilical point $\kappa_1=\kappa_2$ the right-hand side equals $\kappa_1$. The logarithmic mean extends continuously to $[0,\infty)^2$ by assigning the value zero when at least one argument is zero; this agrees with the ray formula because the corresponding integral is then infinite. Hence
\begin{equation}\label{eq:three-dimensional-curvature}
 \Cap_2(\Omega)
 \ge
 \int_{\partial\Omega}
 \frac{\kappa_1-\kappa_2}{\ln\kappa_1-\ln\kappa_2}\,dA
 \ge
 \int_{\partial\Omega}\sqrt{K}\,dA,
\end{equation}
where the last inequality follows from
\[
 \frac{a-b}{\ln a-\ln b}\ge\sqrt{ab},
 \qquad a,b>0.
\]
\end{remark}

\begin{remark}\label{rem:ch-pinching-extension}
Theorem~\ref{thm:intro-gauss-pinching} extends directly to three-dimensional Cartan--Hadamard manifolds. More precisely, let $(N^3,g)$ be Cartan--Hadamard and let $\Omega\Subset N$ be smooth and geodesically convex. If
\[
0<K_-\le K\le K_+,
\qquad
\frac{K_-}{K_+}=\rho\in(0,1],
\]
where $K=\kappa_1\kappa_2$ is the Gauss--Kronecker curvature of $\partial\Omega$, then, for every $1<p<3$,
\[
\frac{\Cap_p(\Omega)^2}{|\partial\Omega|_g^{3-p}}
\ge \Phi_{p,3}(\rho).
\]
More generally, for $n\ge4$ the conclusion of Theorem~\ref{thm:intro-gauss-pinching} remains valid on a Cartan--Hadamard manifold provided that
\[
\int_{\partial\Omega}K\,dA_g\ge |\mathbb S^{n-1}|.
\]
Indeed, the $k=0$ normal-Jacobian comparison and the normal-ray argument give
\[
\Cap_p(\Omega)\ge
\left(\frac{n-p}{p-1}\right)^{p-1}
\int_{\partial\Omega}K^{\frac{p-1}{n-1}}\,dA_g.
\]
When $n=3$, the Gauss equation and $\Sec_N\le0$ imply
\[
K_{\partial\Omega}^{\mathrm{int}}
=K+\Sec_N(T\partial\Omega)\le K.
\]
Since $\partial\Omega\simeq\mathbb S^2$, Gauss--Bonnet yields
\[
\int_{\partial\Omega}K\,dA_g
\ge
\int_{\partial\Omega}K_{\partial\Omega}^{\mathrm{int}}\,dA_g
=4\pi.
\]
Thus the proof of Theorem~\ref{thm:intro-gauss-pinching} applies verbatim, with the Euclidean Gauss-map identity replaced by the preceding inequality. In dimensions $n\ge4$, the stated total Gauss--Kronecker curvature assumption provides exactly the same substitute.
\end{remark}

\section{Cartan--Hadamard capacity--area lower bounds}\label{sec:ch}

Let $(N^n,g)$ be a Cartan--Hadamard manifold and let $\Omega\Subset N$ be smooth and geodesically convex. We recall that, under the uniform upper curvature bound $\Sec_N\le-k^2<0$, $N$ is $p$-hyperbolic for every $p>1$; see \cite[Section~6.2]{DAmbrosioDipierro}. See also Holopainen \cite{Holopainen} for the asymptotic Dirichlet problem for the $p$-Laplacian under negatively pinched curvature. Hence the exterior $p$-capacitary potential exists by the standard $p$-potential theory discussed above. For $p=2$, lower capacity comparisons based on lower bounds for the initial principal curvatures were obtained by Hurtado--Palmer--Ritor\'e \cite{HurtadoPalmerRitore}. For $y\in\partial\Omega$, let $\nu$ be the outward unit normal and consider
\begin{equation}\label{eq:normal-map}
 F(y,s)=\exp_y(s\nu(y)),
 \qquad s\ge0.
\end{equation}
The map $F$ gives global normal coordinates on $N\setminus\Omega$; we denote by $J(y,s)$ the Jacobian of $F$ along the outward normal ray from $y$.

\subsection{The full boundary-curvature estimate}

\begin{proof}[Proof of Theorem~\ref{thm:intro-ch}]
Let $\varphi\in C_c^\infty(N)$ be admissible for $\Cap_p(\Omega)$. In the global normal coordinates \eqref{eq:normal-map},
\begin{equation}\label{eq:ch-energy}
 \int_{N\setminus\Omega}|\nabla\varphi|^p\,dV_g
 \ge
 \int_{\partial\Omega}\int_0^\infty
 |\partial_s\varphi(y,s)|^pJ(y,s)\,ds\,dA(y).
\end{equation}
For each fixed $y$, Lemma~\ref{lem:weighted-ray} gives
\begin{equation}\label{eq:ch-ray}
 \int_0^\infty|\partial_s\varphi|^pJ\,ds
 \ge
 \left(\int_0^\infty J^{-1/(p-1)}\,ds\right)^{1-p}.
\end{equation}
By the normal Jacobi-field comparison theorem of Heintze and Karcher
\cite[Section~3]{HeintzeKarcher}, the Jacobian of the outward normal
exponential map satisfies
\begin{equation}\label{eq:jacobian-comparison-full}
 J(y,s)
 \ge
 \prod_{i=1}^{n-1}
 \left(\cosh(ks)+\frac{\kappa_i(y)}{k}\sinh(ks)\right).
\end{equation}
Since $1-p<0$, \eqref{eq:jacobian-comparison-full} yields
\[
 \left(\int_0^\infty J^{-1/(p-1)}\,ds\right)^{1-p}
 \ge
 \mathcal C_{n,p}(k;\kappa_1(y),\ldots,\kappa_{n-1}(y)).
\]
Integration over $\partial\Omega$ and the infimum over admissible functions prove \eqref{eq:intro-ch-full}.

For \eqref{eq:intro-ch-gauss}, expand the determinant and use the Newton--Maclaurin inequalities directly:
\begin{align*}
 &\prod_{i=1}^{n-1}\left(\cosh(ks)+k^{-1}\kappa_i\sinh(ks)\right)\\
 &\quad=\sum_{j=0}^{n-1}\cosh(ks)^{n-1-j}
 \left(k^{-1}\sinh(ks)\right)^j\sigma_j(\kappa)\\
 &\quad\ge\sum_{j=0}^{n-1}\binom{n-1}{j}\cosh(ks)^{n-1-j}
 \left(k^{-1}\sinh(ks)\right)^jK^{j/(n-1)}\\
 &\quad=\left(\cosh(ks)+k^{-1}K^{1/(n-1)}\sinh(ks)\right)^{n-1}.
\end{align*}
Here $\sigma_j(\kappa)$ denotes the $j$-th elementary symmetric polynomial of the variables $\kappa_1,\ldots,\kappa_{n-1}$; we set $\sigma_0=1$.
Substitution into the defining one-dimensional integral gives
\[
 \mathcal C_{n,p}(k;\kappa_1,\ldots,\kappa_{n-1})
 \ge C_{n,p}(k,K^{1/(n-1)}).
\]
Finally $K^{1/(n-1)}\ge\min_i\kappa_i$ and $\tau\mapsto C_{n,p}(k,\tau)$ is increasing, which proves \eqref{eq:intro-ch}.
\end{proof}

\begin{proof}[Proof of Corollary~\ref{cor:intro-ch-uniform}]
The last line of Theorem~\ref{thm:intro-ch} and the monotonicity of $\tau\mapsto C_{n,p}(k,\tau)$ give
\[
 C_{n,p}\left(k,\min_i\kappa_i(y)\right)
 \ge C_{n,p}(k,\lambda_0)
\]
at every boundary point. Integration proves the conclusion.
\end{proof}

We record several explicit forms of the Cartan--Hadamard model constants, including the three-dimensional logarithmic mean.

\begin{remark}\label{rem:explicit-ch}
At the convex endpoint $\tau=0$, the beta-integral identity gives
\begin{equation}\label{eq:ch-gamma-constant}
 C_{n,p}(k,0)
 =\left[
 \frac{\sqrt\pi}{2k}
 \frac{\Gamma\left(\frac{n-1}{2(p-1)}\right)}
 {\Gamma\left(\frac{n+p-2}{2(p-1)}\right)}
 \right]^{1-p}.
\end{equation}
At the critical exponent $p=\frac{n+1}{2}$,
\begin{equation}\label{eq:ch-critical-constant}
 C_{n,\frac{n+1}{2}}(k,\tau)
 =(k+\tau)^{\frac{n-1}{2}}.
\end{equation}
Consequently, the stronger Gauss-curvature form \eqref{eq:intro-ch-gauss} gives
\begin{equation}\label{eq:ch-critical-pointwise}
 \Cap_{\frac{n+1}{2}}(\Omega)
 \ge
 \int_{\partial\Omega}
 \left(k+K^{1/(n-1)}\right)^{\frac{n-1}{2}}\,dA_g,
\end{equation}
whereas the uniform lower bound $\kappa_i\ge\lambda_0$ gives
\begin{equation}\label{eq:ch-critical-uniform}
 \Cap_{\frac{n+1}{2}}(\Omega)
 \ge
 (k+\lambda_0)^{\frac{n-1}{2}}|\partial\Omega|_g.
\end{equation}

For $n=3$ and $p=2$, the full constant can be evaluated exactly. If $\kappa_1\ne\kappa_2$, the substitution $t=\tanh(ks)$ gives
\begin{align}
 \mathcal C_{3,2}(k;\kappa_1,\kappa_2)
 &=\left[
 \int_0^\infty
 \frac{ds}{\left(\cosh ks+\frac{\kappa_1}{k}\sinh ks\right)
 \left(\cosh ks+\frac{\kappa_2}{k}\sinh ks\right)}
 \right]^{-1}\notag\\
 &=\frac{\kappa_1-\kappa_2}
 {\ln(k+\kappa_1)-\ln(k+\kappa_2)}.
 \label{eq:hyperbolic-log-mean}
\end{align}
At $\kappa_1=\kappa_2$ this is understood by continuity and equals $k+\kappa_1$. Thus
\begin{equation}\label{eq:ch-three-dimensional-full}
 \Cap_2(\Omega)
 \ge
 \int_{\partial\Omega}
 \frac{\kappa_1-\kappa_2}
 {\ln(k+\kappa_1)-\ln(k+\kappa_2)}\,dA_g
 \ge
 \int_{\partial\Omega}\sqrt{(k+\kappa_1)(k+\kappa_2)}\,dA_g.
\end{equation}
In particular,
\begin{equation}\label{eq:ch-three-dimensional-gauss}
 \Cap_2(\Omega)
 \ge
 \int_{\partial\Omega}(k+\sqrt K)\,dA_g,
\end{equation}
since $\sqrt{(k+\kappa_1)(k+\kappa_2)}\ge k+\sqrt{\kappa_1\kappa_2}$.
\end{remark}

We next describe the equality model in the constant-curvature space form.

\begin{remark}\label{rem:ch-model}
In the space form of constant sectional curvature $-k^2$, the boundary of a geodesic ball of radius $R$ has principal curvature $k\coth(kR)$. Its outward normal Jacobian is
\[
 \left(\frac{\sinh(k(R+s))}{\sinh(kR)}\right)^{n-1}
 =\left(\cosh(ks)+\coth(kR)\sinh(ks)\right)^{n-1},
\]
which is exactly the model with $\tau=k\coth(kR)$. The radial capacitary potential is the extremal function in Lemma~\ref{lem:weighted-ray}, and equality holds. As $R\to\infty$, the capacity--area ratio tends to
\[
 \left(\frac{(n-1)k}{p-1}\right)^{p-1}.
\]
\end{remark}

\begin{proof}[Proof of the rigidity statement in Corollary~\ref{cor:intro-ch-uniform}]
The warped-product model directly gives equality. Conversely, suppose equality holds and apply the lower-bound argument to the minimizing potential $u$. Since $\kappa_i(y)\ge\lambda_0$ and the full constant is strictly increasing in each principal curvature,
\[
 \mathcal C_{n,p}(k;\kappa_1(y),\ldots,\kappa_{n-1}(y))
 \ge
 \mathcal C_{n,p}(k;\lambda_0,\ldots,\lambda_0)
 =C_{n,p}(k,\lambda_0)
\]
pointwise on $\partial\Omega$. The difference between the two sides is continuous and nonnegative, while equality in the integrated capacity estimate makes its integral zero. Hence the difference vanishes pointwise, and strict monotonicity gives
\[
 \kappa_1(y)=\cdots=\kappa_{n-1}(y)=\lambda_0
\]
for every $y\in\partial\Omega$. Equality in
\[
 |\nabla u|^p\ge|\partial_su|^p
\]
forces the tangential gradient of $u$ to vanish, so $u$ depends only on $s$. Equality in the ray inequality and in the Jacobian comparison then gives, first almost everywhere and subsequently everywhere by continuity,
\[
 J(y,s)=
 \left(\cosh(ks)+\frac{\lambda_0}{k}\sinh(ks)\right)^{n-1},
\]
and $u$ is the one-dimensional extremal \eqref{eq:intro-ch-rigidity-potential}.

The equality discussion in the same normal Jacobi-field comparison theorem
\cite[Section~3.4.8]{HeintzeKarcher} then shows that equality in the
Jacobian comparison along every normal ray forces the normal Jacobi fields
to agree with the constant-curvature model. Since
$\kappa_1=\cdots=\kappa_{n-1}=\lambda_0$ on $\partial\Omega$, this gives
\[
 g
 =
 ds^2+
 \left(\cosh(ks)+\frac{\lambda_0}{k}\sinh(ks)\right)^2
 g_{\partial\Omega}
\]
on the exterior, and all radial sectional curvatures are equal to $-k^2$.
This proves \eqref{eq:intro-ch-rigidity-metric}. In the space form, a compact
totally umbilical hypersurface is a geodesic sphere.
\end{proof}

The following gives a simple sufficient curvature condition for the warped equality model.

\begin{remark}\label{rem:ch-model-curvature}
Conversely, let $(\Sigma^{n-1},g_\Sigma)$ be a closed Riemannian manifold and set
\[
 f(s)=\cosh(ks)+\frac{\lambda_0}{k}\sinh(ks).
\]
The warped product $ds^2+f(s)^2g_\Sigma$ satisfies $\Sec\le-k^2$ provided
\begin{equation}\label{eq:ch-boundary-curvature-condition}
 \Sec_{g_\Sigma}\le\lambda_0^2-k^2
\end{equation}
on every tangent two-plane of $\Sigma$ (the condition is vacuous when $n=2$). Indeed, every radial sectional curvature equals $-f''/f=-k^2$, whereas a tangent two-plane $\Pi\subset T\Sigma$ has
\[
 \Sec_g(\Pi)
 =\frac{\Sec_{g_\Sigma}(\Pi)-(f'(s))^2}{f(s)^2}
 \le -k^2,
\]
because $(f')^2-k^2f^2=\lambda_0^2-k^2$. Thus \eqref{eq:ch-boundary-curvature-condition} is a simple intrinsic sufficient condition ensuring that the warped end itself satisfies the sectional-curvature bound required in Theorem~\ref{thm:intro-ch}. No assertion about filling this end by a complete Cartan--Hadamard manifold is needed here.
\end{remark}

\begin{remark}
The proof uses only the ambient sectional-curvature bound and the second fundamental form of the initial boundary. It does not require convexity or horoconvexity of the capacitary level hypersurfaces.
\end{remark}

\section*{Acknowledgements}
The author was supported by the Fundamental Research Funds for the Central
Universities, HUST (No.~2025BRSXB002), and the National Natural Science
Foundation of China (Grant No.~12471054).

\section*{Statements and Declarations}
\noindent\textbf{Competing interests.}
The author declares no competing interests.

\medskip
\noindent\textbf{Data availability.}
Data sharing is not applicable to this article, as no datasets were generated
or analysed during the current study.

\medskip
\noindent\textbf{Use of AI-assisted tools.}
The mathematical results and proof strategies in this manuscript are the
responsibility of the author. ChatGPT 5.6 Sol was used as an assistive tool in
exploring and checking the endpoint rigidity argument in the proof of
Theorem~\ref{thm:intro-gauss-comparison}, as well as the specific material in
Example~\ref{ex:ellipsoid-obstructions} and Remark~\ref{rem:explicit-ch}.
It was also used for  improving organization and assisting with LaTeX typesetting. All AI-assisted
mathematical content, calculations, references, and expository revisions were
independently reviewed and verified by the author, who takes full responsibility
for the correctness, originality, and final content of the manuscript.

\end{document}